\documentclass[12pt]{elsarticle}

\usepackage{hyperref}
\usepackage{amssymb}
\usepackage{amsthm,bm,mathtools,amsfonts}
\usepackage{lineno}
\usepackage{mathrsfs}
\usepackage{amsmath}
\usepackage{xcolor}
\usepackage{cases}
\newtheorem{theorem}{Theorem}[section] 
\newtheorem{lemma}[theorem]{Lemma}     
\newtheorem{corollary}[theorem]{Corollary}
\newtheorem{proposition}[theorem]{Proposition}
\newtheorem{conjecture}[theorem]{Conjecture}

\numberwithin{equation}{section}

\newcommand{\lfr}[1]{\left\lfloor #1 \right\rfloor}

\makeatletter
\def\ps@pprintTitle{%
	\let\@oddhead\@empty
	\let\@evenhead\@empty
	\let\@oddfoot\@empty
	\let\@evenfoot\@empty
}
\makeatother

\begin{document}

\begin{frontmatter}

\title{The real-rootedness of the toric $g$-contribution polynomials}

\author[]{Qiqi Xiao\corref{cor1}}
\ead{xiaoqqcs@hotmail.com}
\cortext[cor1]{Corresponding author}
\address{School of Mathematical Sciences, Dalian University of Technology, Dalian 116024,\\ P. R. China}

\begin{abstract}
Recently, Ehrenborg, Hetyei and Readdy expressed the toric $g$-polynomial of a simple polytope 
as a linear combination of a family of polynomials, 
called $g$-contribution polynomials, 
with coefficients given by the entries of its gamma-vector.
They conjectured that these toric $g$-contribution polynomials are real-rooted.
This paper proves this conjecture.
\end{abstract}


\end{frontmatter}
\section{Introduction}
The toric $g$-vector is an important enumerative invariant associated with Eulerian posets \cite[Section 3.16]{Sta12}.
A fundamental property of the toric $g$-polynomial is the nonnegativity of its coefficients. 
For face lattices of rational convex polytopes, 
Stanley established the nonnegativity  of the toric $g$-polynomial coefficients 
by applying the hard Lefschetz theorem to the intersection homology of projective toric varieties \cite[ Corollary 3.2]{Sta87}. 
This property was subsequently extended to general convex polytopes by Karu, utilizing the hard Lefschetz theorem for combinatorial intersection homology \cite{Kar04}. 
Ehrenborg, Hetyei, and Readdy \cite{EHR25} showed that 
the toric $g$-polynomial of a simple polytope 
can be expressed as a linear combination of a family of polynomials, 
called $g$-contribution polynomials, 
with coefficients given by the entries of its $\gamma$-vector.
Furthermore,
they gave a closed form formula for the toric $g$-vector of the permutahedron and of several other simple polytopes, and found a combinatorial interpretation in each case.

An $n$-dimensional polytope $P$ is \emph{simple} if every vertex is incident to $n$ facets (maximal  proper faces).
The associated $h$-polynomial is given by
\[
\sum_{i=0}^n h_i x^i = \sum_{i=0}^n f_i (x-1)^i,
\]
where $f_i$ is the number of $i$-dimensional faces of $P$.
By the Dehn--Sommerville relations this polynomial is palindromic, and hence it can be
written uniquely in the form
\[
h(P,x)=\sum_{j=0}^{\lfloor n/2\rfloor}
\gamma_j x^j(1+x)^{n-2j}.
\]
The vector $(\gamma_0,\gamma_1,\ldots,\gamma_{\lfloor n/2\rfloor})$ is called the $\gamma$-vector of $P$; 
it was introduced by Gal \cite{Gal05}.

Following the notation of Ehrenborg et. al \cite{EHR25}, for $n\ge j$, we define the toric $g$-contribution polynomials by
\begin{equation}\label{def:gnj}
g_{n,j}(x)=\sum_{k=0}^{\min(\lfr{n/2},n-j)}
C_{\,n-k-j}\binom{n-k}{k}(x-1)^k,
\end{equation}
where $C_r=\frac{1}{r+1}\binom{2r}{r}$ is the $r$-th Catalan number.
The polynomial $g_{n,0}(x)$ is the toric $g$-polynomial of the $n$-dimensional cube \cite{Het12}.
Ehrenborg et. al expressed the toric $g$-polynomial of any simple polytope 
as a linear combination of these toric $g$-contribution polynomials. 

\begin{theorem}(\cite[Theorem 3.4]{EHR25})
Let $P$ be an $n$-dimensional simple polytope with $\gamma$-vector
$(\gamma_0,\gamma_1,\ldots,\gamma_{\lfr{n/2}})$.  
Then,
\[
g(P,x)=\sum_{j=0}^{\lfloor n/2\rfloor}\gamma_j\, g_{n,j}(x).
\]
\end{theorem}

The polynomial $g_{n,j}(x)$ gives the contribution of the entry $\gamma_j$ to the toric $g$-polynomial of an $n$-dimensional simple polytope.
These polynomials are independent of the particular polytope and depend only on $n$ and $j$. 

Ehrenborg et. al \cite{EHR25} used this formula to study the toric
$g$-vectors of several families of nestohedra, including associahedra, cyclohedra, permutahedra, and chordal nestohedra.  
Their work also led to the following real-rootedness conjecture for $g_{n,j}(x)$, where $0\le j\le \lfr{n/2}$.

\begin{conjecture} (\cite[Conjecture 11.1]{EHR25}) \label{Conj_g}
The toric $g$-contribution polynomials $g_{n,j}(x)$ are real-rooted for $0\le j\le \lfloor n/2 \rfloor$.
\end{conjecture}

The main result of this paper proves this conjecture.

\begin{theorem}\label{thm:main}
For all integers $n\geq 0$ and $0\leq j\leq \lfr{n/2}$, the polynomial $g_{n,j}(x)$ is real-rooted.
Moreover, $g_{n,j}(x)$ interlaces $g_{n+1,j}(x)$ for $0\leq j\leq \lfr{n/2}$,
and $g_{n,j}(x)$ interlaces $g_{n+1,j+1}(x)$ for $0\leq j\leq \lfr{(n-1)/2}$.

\end{theorem}

The paper is organized as follows.  
Section~2 provides a new recurrence relation for the toric $g$-contribution polynomials. 
Section~3 proves Theorem \ref{thm:main} and discusses consequences for the coefficient sequences of $g_{n,j}(x)$.  
The final section indicates some related questions suggested by the Ehrenborg--Hetyei--Readdy expansion.

\section{Recurrence relation}

The toric $g$-contribution polynomial $g_{n,j}(x)$ is defined in equation \eqref{def:gnj},
the main conjecture concerns the subrange $0\le j\le \lfloor n/2\rfloor$.

For $0\le j\le \lfloor n/2\rfloor$, Ehrenborg et al. \cite[Lemma 3.3]{EHR25} showed that 
the coefficient of $x^k$ in $g_{n,j}(x)$ equals 
the number of Dyck paths of semilength $n-j$ with $k$ peaks 
whose first coordinate is at most $n-1$. 
In particular, $g_{n,j}(x)$ has nonnegative coefficients.
Therefore,
\begin{equation}\label{g0>0}
	g_{n,j}(x)>0 \qquad \text{for } x\in(0,\infty),
\end{equation}
whenever $0\le j\le \lfr {n/2}$.
They also established a recurrence relation for $g_{n,j}(x)$. 
We restate this recurrence relation here.

\begin{lemma}(\cite[Lemma 5.1]{EHR25})\label{rec_g}
The toric $g$-contribution polynomials $g_{n,j}(x)$ satisfy the recurrence
\begin{equation}\label{rec_g3}
g_{n,j}(x) = g_{n-1,j-1}(x) + (x - 1) \cdot g_{n-2,j-1}(x)
\end{equation}
for $n \ge 2$ and $j \ge 1$, with $g_{0,0}(x)=g_{1,0}(x)=1$.    
\end{lemma}

Since recurrence \eqref{rec_g3} alone fails to establish the real-rootedness of the polynomial $g_{n,j}(x)$, 
it is reasonable to explore other recurrence relations. 
In what follows, we provide a recurrence relation from which the real-rootedness of the polynomial $g_{n,j}(x)$ can be deduced. 
We first require the following result.

Let $C_{n}=\frac{1}{n+1}\binom{2n}{n}$ be the $n$-th Catalan number, 
and let
\begin{equation}\label{eq:Cu}
C(u)=\sum_{n=0}^{\infty}C_n u^n = \frac{1 - \sqrt{1 - 4u}}{2u}
\end{equation}
be the Catalan generating function.

\begin{lemma}\label{rec_catalan}
The Catalan generating function $C(u)$ satisfies the differential equation
\begin{equation}\label{eq:C'C}
u(1-4u)C'(u)+(1-2u)C(u)=1.
\end{equation}
\end{lemma}
\begin{proof}
It is well known that the generating function of Catalan numbers $C(u)$ satisfies the functional equation \cite[p. 192]{Ath26}
\begin{equation}\label{eq:C}
C(u)=1+uC(u)^2.
\end{equation}
Differentiating both sides of \eqref{eq:C} with respect to $u$ and 
solving for $C'(u)$, we get
\begin{equation}\label{eq:C'}
C'(u)=\frac{C(u)^2}{1-2uC(u)}.
\end{equation}

Substituting equation \eqref{eq:C'} into the left-hand side (LHS) of the identity equation \eqref{eq:C'C} and using again \eqref{eq:C}, 
we get
\[
\text{LHS}=\frac{(1-2u)C(u)-uC(u)^2}{1-2uC(u)}=\frac{1-2uC(u)}{1-2uC(u)}=1=\text{RHS}.
\]
The identity is proved.
\end{proof}

Let 
\begin{equation}\label{eq:f-def}
f_{n,j}(x)\coloneqq g_{n,j}(x+1)
=
\sum_{k=0}^{\min\{\lfloor n/2\rfloor,\,n-j\}}
C_{n-k-j}\binom{n-k}{k}x^k.
\end{equation}
With the convention $C_{r}=0$ for $r<0$, 
we have $f_{n,j}(x)=0$ whenever $n<j$.
Clearly,
the toric $g$-contribution polynomial $g_{n,j}(x)$ is real-rooted if and only if
$f_{n,j}(x)$ is real-rooted.

By Lemma \ref{rec_catalan}, 
we can establish the following recurrence relation for $f_{n,j}(x)$. 

\begin{proposition}\label{rec:fnjr}
For all integers $n \ge j$, 
\[
\begin{aligned}
(n + 2 - j) f_{n+1,j}(x)
={}& \big[(x+4)n - 4j + 2\big] f_{n,j}(x) - 2x(x+1) f_{n,j}'(x) \\
&+ 2x(n - 2j) f_{n-1,j}(x) + r_{n+1,j}(x)
\end{aligned}
\]
with $f_{j,j}(x) = 1$,
where
\[
r_{n+1,j}(x) =
\begin{cases}
\dbinom{j}{n+1-j} x^{n+1-j}, & \text{if } j-1 \le n \le 2j-1, \\[6pt]
0, & \text{otherwise}.
\end{cases}
\]

In particular, for $n \ge 2j$,
\begin{equation}\label{rec:fnj}
\begin{aligned}
	(n + 2 - j) f_{n+1,j}(x)
	={}& \big[(x+4)n - 4j + 2\big] f_{n,j}(x) - 2x(x+1) f_{n,j}'(x) \\
	&+ 2x(n - 2j) f_{n-1,j}(x).
\end{aligned}
\end{equation}
\end{proposition}

\begin{proof}
Let
\[
F_j(x,t) = \sum_{n \ge 0} f_{n,j}(x) \, t^n
\]
be the generating function of $f_{n,j}(x)$ in $t$. 
Using the convention $C_m=0$ for $m<0$ and setting $m=n-k-j$, we obtain
\begin{align*}
	F_j(x,t)
	&= \sum_{m\ge 0}\sum_{k=0}^{m+j}
	C_m\binom{m+j}{k}x^k t^{m+j+k} \\
	&= t^j \sum_{m\ge 0} C_m t^m
	\sum_{k=0}^{m+j}\binom{m+j}{k}(xt)^k \\[0.25cm]
	&= t^j(1+xt)^j\sum_{m\ge 0} C_m\bigl[t(1+xt)\bigr]^m .
\end{align*}

Let $u = t(1+xt)$, then
\begin{equation}\label{eq:FCu}
F_j(x,t) = u^j \cdot C(u).
\end{equation}
Let $G_j(u) \coloneqq F_j(x,t)=u^j C(u)$, so that 
\begin{equation}\label{eq:CG}
C(u) = u^{-j} G_j(u).   
\end{equation}
Differentiating both sides of equation \eqref{eq:CG} with respect to $u$, we obtain
\begin{equation}\label{eq:C'G'}
C'(u) = u^{-j} G_j'(u) - j u^{-j-1} G_j(u).
\end{equation}
Now substitute both equations \eqref{eq:CG} and \eqref{eq:C'G'} into the equation \eqref{eq:C'C}, we have
\begin{equation}\label{G'G}
u(1-4u) G_j'(u) + \bigl[(1-j) + (4j-2)u\bigr] G_j(u) = u^j.
\end{equation}

Since $F_j(x,t) = G_j(u)$ and $u = t(1+xt)$, 
the chain rule gives
\begin{equation}\label{eq:ft}
\frac{\partial F_j(x,t)}{\partial t} = G_j'(u) \cdot \frac{\partial u}{\partial t}=(1+2xt) G_j'(u),
\end{equation}
and 
\begin{equation}\label{eq:fx}
\frac{\partial F_j(x,t)}{\partial x} = G_j'(u) \cdot \frac{\partial u}{\partial x}= t^2 G_j'(u).
\end{equation}
Then,
\begin{align}\label{eq:G'F}
u(1-4u) G_j'(u) &= t(1+xt)\left[1-4t(1+xt)\right] G_j'(u) \nonumber\\
&= \left[ \bigl(t - (x+4)t^2 - 2xt^3\bigr)(1+2xt) + 2x(x+1)t^3 \right] G_j'(u) \nonumber\\
&= \bigl[t - (x+4)t^2 - 2xt^3\bigr] \frac{\partial F_j(x,t)}{\partial t} + 2x(x+1)t \frac{\partial F_j(x,t)}{\partial x}
\end{align}

Now replace $u$ with $t(1+xt)$ and substitute equation \eqref{eq:G'F} into equation \eqref{G'G}. 
We obtain the equation
\begin{align}\label{eq:Fxt}
&\bigl[t - (x+4)t^2 - 2xt^3\bigr] \frac{\partial F_j(x,t)}{\partial t}
+ 2x(x+1)t \frac{\partial F_j(x,t)}{\partial x} \nonumber \\
+&\ \bigl[(1-j) + (4j-2)t + (4j-2)xt^2\bigr] F_j(x,t)
= t^j(1+xt)^j.
\end{align}

We now extract the coefficient of $t^{n+1}$ from both sides of \eqref{eq:Fxt}.
For the right-hand side of $t^j (1 + xt)^j$, we have
\begin{align}\label{coe:R}
[t^{n+1}]t^j (1+xt)^j = [t^{n+1-j}](1+xt)^j=\binom{j}{n+1-j}x^{n+1-j}.
\end{align}
The coefficient of $t^{n+1}$ is exactly $r_{n+1,j}(x)$ as defined in the proposition.

We now compute the left-hand side coefficients for $t^{n+1}$.
Note that 
\[
\frac{\partial F_j(x,t)}{\partial t} = \sum_{n \ge 0} (n+1)f_{n+1,j}(x) \, t^n, \qquad
\frac{\partial F_j(x,t)}{\partial x} = \sum_{n \ge 0} f_{n,j}'(x) \, t^n.
\]
Then, the coefficient of $t^{n+1}$ in the left-hand side of equation \eqref{eq:Fxt} is 
\begin{align}\label{coe:L}
& (n + 2 - j) f_{n+1,j}(x)- \big[(x+4)n - 4j + 2\big] f_{n,j}(x) \nonumber\\
+&\  2x(x+1) f_{n,j}'(x) - 2x(n - 2j) f_{n-1,j}(x). 
\end{align}
Combining \eqref{coe:R} and \eqref{coe:L} yields the recurrence
\[
\begin{aligned}
(n + 2 - j) f_{n+1,j}(x)
={}& \big[(x+4)n - 4j + 2\big] f_{n,j}(x)- 2x(x+1) f_{n,j}'(x) \\
&+ 2x(n - 2j) f_{n-1,j}(x)+ r_{n+1,j}(x).
\end{aligned}
\]
For $n \ge 2j$, we have $n+1 > 2j$, so $r_{n+1,j}(x) = 0$.
This completes the proof.
\end{proof}

Replacing $x$ by $x-1$ in \eqref{rec:fnj} and using the identity $f_{n,j}'(x-1)=g_{n,j}'(x)$ we obtain the following recurrence relation for $g_{n,j}(x)$,
which plays a key role in proving the real-rootedness of $g_{n,j}(x)$ for $0\le j\le \lfr{n/2}$.
Note that $j\le \lfr{n/2}$ implies $n\ge 2j$.
This restriction is essential: after extending the definition to $0\le j\le n$, 
the polynomial $g_{n,j}(x)$ need not be real-rooted when $n<2j$. 
For example,
\[
g_{5,3}(x)=3x^2-2x+1,\quad
g_{6,4}(x)=6x^2-7x+3,\quad
g_{7,4}(x)=4x^3-2x^2+4x-1
\]
are not real-rooted. 
Thus, for the conjecture of Ehrenborg et. al, 
one has to consider the range $n\ge 2j$.

\begin{corollary}\label{cor_gnj}
Let $n\ge 2j$ be an integer.
The toric $g$-contribution polynomials $g_{n,j}(x)$ satisfy the recurrence relation 
\begin{equation}\label{rec_gnj}
\begin{aligned}
	(n + 2 - j) g_{n+1,j}(x)
	={}& \big[(x+3)n - 4j + 2\big] g_{n,j}(x) - 2x(x-1) g_{n,j}'(x) \\
	&+ 2(x-1)(n - 2j) g_{n-1,j}(x),\quad n\ge 2j+1,
\end{aligned}  
\end{equation}
where 
\[
g_{2j,j}(x)=\sum_{k=0}^{j}
C_{j-k}\binom{2j-k}{k}(x-1)^k,
\]
and 
\[
g_{2j+1,j}(x)=\sum_{k=0}^{j}
C_{j-k+1}\binom{2j-k+1}{k}(x-1)^k.\]
\end{corollary}

\section{Real-rootedness}
A polynomial $f(x)$ with real coefficients is called \textit{real-rooted} if every root of $f(x)$ is real, or $f(x) \equiv 0$.

A real-rooted polynomial $g(x)$, with roots $\cdots \le s_2 \le s_1$, is said to \emph{interlace} a real-rooted polynomial $f(x)$, with roots $\cdots \le r_2 \le r_1$, if
\begin{align}\label{root:fg}
\cdots \le s_2 \le r_2 \le s_1 \le r_1.
\end{align}
We then write $g(x) \preceq f(x)$, see, for instance \cite[Section 7.8]{Bra15}. 
If no equality sign occurs in \eqref{root:fg}, 
then we say that $g(x)$ \emph{strictly interlaces} $f(x)$, denoted by $g(x) \prec f(x)$. For notational convenience, 
let $a \preceq bx + c$ for any real constants $a,b,c$ 
and $f(x) \preceq 0$, $0 \preceq f(x)$ for any real-rooted polynomial $f(x)$.

The following standard lemma will be applied several times in this paper.

\begin{lemma}(\cite[Theorem 2.3]{LW07})\label{rz:F}
Let $F, f, g_1, \dots, g_k$ be real polynomials satisfying the following conditions.
\begin{enumerate}
\item[(a)] $F(x) = a(x)f(x) + b_1(x)g_1(x) + \cdots + b_k(x)g_k(x)$, 
where $a(x), b_1(x), \dots, b_k(x)$ are real polynomials, 
such that $\deg F = \deg f$ or $\deg f + 1$.
\item[(b)] $f, g_j$ are real-rooted and $g_j \preccurlyeq f$ for each $j$.
\item[(c)] $F$ and $g_1, \dots, g_k$ have leading coefficients of the same sign.
\end{enumerate}

Suppose that $b_j(r) \le 0$ for each $j$ and each zero $r$ of $f$. 
Then, $F$ is real-rooted and $f \preccurlyeq F$. 
In particular, if for each zero $r$ of $f$
there exists an index $j$ such that $g_j \prec f$ and $b_j(r) < 0$, then $f \prec F$.
\end{lemma}

Before proving that $g_{n,j}(x)$ is real-rooted for $n\ge 2j$, 
we first consider the real-rootedness of $f_{2j,j}(x)$ and $f_{2j+1,j}(x)$. 
For convenience, we record their explicit formulas here.
\[
f_{2j,j}(x)=\sum_{k=0}^j C_{j-k}\binom{2j-k}{k}x^k,\quad
f_{2j+1,j}(x)=\sum_{k=0}^j C_{j-k+1}\binom{2j-k+1}{k}x^k.
\]
For small values of $j$, 
these polynomials are as follows:
$$f_{0,0}(x)=1, \quad f_{2,1}(x)=1+x, \quad f_{4,2}(x)=2+3x+x^2=(x+1)(x+2)$$ 
and 
$$f_{1,0}(x)=1, \quad f_{3,1}(x)=2+2x, \quad f_{5,2}(x)=5+8x+3x^2=(x+1)(3x+5).$$

\begin{lemma}\label{pj<qj}
Let $j\ge 0$ be an integer.
The polynomials $f_{2j,j}(x)$ and $f_{2j+1,j}(x)$ are both real-rooted and
$f_{2j,j}(x)\preceq f_{2j+1,j}(x)$.
\end{lemma}
\begin{proof}
For convenience, let $p_j(x) \coloneqq f_{2j,j}(x)$  and $q_j(x) \coloneqq f_{2j+1,j}(x)$.  
It is obvious that $\deg(p_j(x))=\deg(q_j(x))=j$.
We first consider the recurrence relations for $p_j(x)$ and $q_j(x)$.

After the change of variables $f_{n,j}(x)=g_{n,j}(x+1)$, the recurrence \eqref{rec_g3} becomes
\[
f_{n,j}(x)=f_{n-1,j-1}(x)+x f_{n-2,j-1}(x).
\]
Taking $n=2j$, we get
\[
f_{2j,j}(x)=f_{2j-1,j-1}(x)+x f_{2j-2,j-1}(x)
=f_{2(j-1)+1,j-1}(x)+x f_{2(j-1),j-1}(x),
\]
which implies that $p_j(x)$ satisfies the recurrence
\begin{equation}\label{rec:pj} 
p_j(x) = q_{j-1}(x) + xp_{j-1}(x), \quad j \ge 1, 
\end{equation}
with $p_0(x)=q_0(x)=1$.

Also, by direct calculation, we have
\begin{align*}
(j+2) C_{j-k+1}\binom{2j-k+1}{k} &=  \frac{(j+2)(2j-2k+2)!}{(j-k+2)!(j-k+1)!} \cdot \frac{(2j-k+1)!}{k!(2j-2k+1)!} \\
&= \frac{(j+2)(2j-k+1)}{j-k+2} \cdot \frac{2(2j-k)!}{k!(j-k)!(j-k+1)!}\\
&= \left[ (2j+1) + \frac{k(j-1)}{j-k+2}\right] \cdot \frac{2(2j-k)!}{k!(j-k)!(j-k+1)!}.
\end{align*}
Note that 
\begin{align*} 
(2j+1)\frac{2(2j-k)!}{k!(j-k)!(j-k+1)!} 
&= 2(2j+1) \left[ \frac{(2j-2k)!}{(j-k+1)!(j-k)!} \cdot \frac{(2j-k)!}{k!(2j-2k)!} \right] \\
&= 2(2j+1) C_{j-k}\binom{2j-k}{k}
\end{align*}
and 
\begin{align*}
\frac{k(j-1)}{j-k+2} \cdot \frac{2(2j-k)!}{k!(j-k)!(j-k+1)!} 
&= (j-1)C_{j-k+1}\binom{2j-k}{k-1}.
\end{align*}
Thus,
\begin{equation*} 
(j+2) C_{j-k+1}\binom{2j-k+1}{k}=2(2j+1) C_{j-k}\binom{2j-k}{k}+(j-1)C_{j-k+1}\binom{2j-k}{k-1}.
\end{equation*}
Multiplying this identity by $x^k$ and summing over $0\le k\le j$, we obtain the recurrence
\begin{equation}\label{rec:qj}
(j + 2)q_j(x) = 2(2j + 1)p_j(x) + (j - 1)xq_{j-1}(x), \quad j \ge 1 
\end{equation}
with $p_0(x)=q_0(x)=1$.

Combining \eqref{rec:pj} and \eqref{rec:qj}, 
we obtain the following three-term recurrences:
\begin{align}\label{eq:p_expanded}
(j + 2)p_{j+1}(x) &= (2j + 1)(x + 2)p_j(x) - (j - 1)x^2 p_{j-1}(x), \quad j\ge 1
\end{align}
with $p_0(x)=1, p_1(x)=1+x$,
and 
\begin{align}\label{eq:q_expanded}
(j + 3)(2j + 1)q_{j+1}(x) 
={}& 2\big[(2j+3)(2j+1) + (2j^2 + 4j + 3)x\big]q_j(x) \nonumber \\
& - (j - 1)(2j + 3)x^2 q_{j-1}(x), \quad j\ge 1
\end{align}
with $q_0(x)=1, q_1(x)=2+2x$.

We prove that $p_j(x)$ is real-rooted by induction on $j$.
The initial case is $p_0(x)=1\preceq p_1(x)=1+x$. 
Assume that $p_{j-1}(x)\preceq p_j(x)$ and that both polynomials are real-rooted. 
From \eqref{eq:p_expanded},
\[
p_{j+1}(x)
=\frac{(2j+1)(x+2)}{j+2}p_j(x)
-\frac{(j-1)x^2}{j+2}p_{j-1}(x),\qquad j\ge 1.
\]
Let $r$ be a zero of $p_j(x)$. Then
\[
-\frac{(j-1)r^2}{j+2}\le 0.
\]
Moreover, $\deg p_{j+1}=\deg p_j+1$, and all polynomials involved have positive leading coefficients. 
Therefore, Lemma~\ref{rz:F}  implies that $p_{j+1}(x)$ is real-rooted and that
\[
p_j(x)\preceq p_{j+1}(x).
\]
The proof for $q_j(x)$ is analogous to that for $p_j(x)$ and is omitted.

It remains to prove $p_j(x)\preceq q_j(x)$. The cases $j=0$ and $j=1$ are immediate, since $p_0=q_0=1$ and $q_1=2p_1$. For $j\ge 2$, \eqref{rec:pj} and \eqref{rec:qj} give
\[
(j+2)q_j(x)=\big[(j-1)x+2(2j+1)\big]p_j(x)-(j-1)x^2p_{j-1}(x), \qquad j\ge 1.
\]
By the result already proved, $p_{j-1}(x)\preceq p_j(x)$. For every zero $r$ of $p_j(x)$,
\[
-(j-1)r^2\le 0.
\]
Applying Lemma~\ref{rz:F} once more gives
\[
p_j(x)\preceq q_j(x).
\]
This completes the proof.

\end{proof}

We are now ready to prove Theorem~\ref{thm:main}.

\begin{proof}[Proof of Theorem~\ref{thm:main}.]

Fix $j\ge 0$. Let $n\ge 2j$. 
We proceed by induction on  $n$.
The initial cases $g_{2j,j}(x)$ and  $g_{2j+1,j}(x)$ are both real-rooted and
\[
g_{2j,j}(x)\preceq g_{2j+1,j}(x),
\]
which follows from Lemma~\ref{pj<qj}. 
Assume now that $n\ge 2j+1$ and that
\[
g_{n-1,j}(x)\preceq g_{n,j}(x),
\]
with both polynomials real-rooted. We rewrite \eqref{rec_gnj} as
\[
g_{n+1,j}(x)
=a_n(x)g_{n,j}(x)+b_n(x)g_{n,j}'(x)+c_n(x)g_{n-1,j}(x),
\]
where
\[
a_n(x)=\frac{nx+3n-4j+2}{n+2-j},\quad
b_n(x)=-\frac{2x(x-1)}{n+2-j},\quad
c_n(x)=\frac{2(x-1)(n-2j)}{n+2-j}.
\]
If $g_{n,j}'(x)\not\equiv 0$, Rolle's theorem gives
\[
g_{n,j}'(x)\preceq g_{n,j}(x).
\]
If $g_{n,j}'(x)\equiv 0$, the derivative term is absent. 
Let $r$ be any zero of $g_{n,j}(x)$. 
Since $g_{n,j}(x)$ has nonnegative coefficients and is real-rooted, $r\le 0$. 
Hence,
\[
b_n(r)=-\frac{2r(r-1)}{n+2-j}\le 0,\qquad
c_n(r)=\frac{2(r-1)(n-2j)}{n+2-j}\le 0,
\]
because $n\ge 2j+1$. 
Also $\deg g_{n+1,j}=\deg g_{n,j}$ or $\deg g_{n,j}+1$ follows from the definition
and all nonzero polynomials involved have positive leading coefficients. 
Lemma~\ref{rz:F} implies that $g_{n+1,j}(x)$ is real-rooted and $g_{n,j}(x)\preceq g_{n+1,j}(x)$.
This proves that
\[
g_{n,j}(x)\preceq g_{n+1,j}(x)
\]
for all $0\le j\le \lfloor n/2\rfloor$.

It remains to prove the second interlacing relation. 
Let $ 0\le j\le \lfr{(n-1)/2}$, so that $n\ge 2j+1$. 
Applying \eqref{rec_g3} with $n+1$ and $j+1$ in place of $n$ and $j$, respectively, gives
\[
g_{n+1,j+1}(x)=g_{n,j}(x)+(x-1)g_{n-1,j}(x).
\]
By the first part of the proof,
\[
g_{n-1,j}(x)\preceq g_{n,j}(x).
\]
Let $r$ be any zero of $g_{n,j}(x)$. 
As above, $r\le 0$, and hence $ r-1<0$.
Moreover,
$\deg g_{n+1,j+1}=\deg g_{n,j}$ or $\deg g_{n,j}+1$.
Applying Lemma~\ref{rz:F} 
we obtain
\[
g_{n,j}(x)\preceq g_{n+1,j+1}(x).
\]
The proof is complete.
\end{proof}

Furthermore, we have the following result.

\begin{proposition}\label{prop:g1}
Let $j\ge 0$. For $n \ge 2j$, 
\begin{equation*}
g_{n,j}(0) = \begin{cases}
	1, & \text{if } j = 0, \\
	0, & \text{if } j \ge 1.
\end{cases}
\end{equation*}
\end{proposition}
\begin{proof}
Since \(g_{n,j}(x)=f_{n,j}(x-1)\), it suffices to show that
$f_{n,j}(-1)=1$ for $j=0$ and $f_{n,j}(-1)=0$ for $j\ge 1$.
Evaluating \eqref{eq:FCu} at \(x=-1\), 
we get
\begin{equation*}
	F_j(-1,t)=t^j(1-t)^j C\bigl(t(1-t)\bigr).
\end{equation*}
Substituting \(u=t(1-t)\) into \eqref{eq:Cu} we get, 
as an identity of formal power series,
\[
C\bigl(t(1-t)\bigr)=(1-t)^{-1}.
\]
Hence,
\begin{equation*}
F_j(-1,t)=t^j(1-t)^j(1-t)^{-1}=t^j(1-t)^{j-1}.
\end{equation*}
The desired value $f_{n,j}(-1)$ is precisely the coefficient of $t^n$ in the expansion of $F_j(-1, t)$. 
We consider two cases based on the value of $j$:

If $j = 0$, 
then $F_0(-1,t) = t^0(1-t)^{-1} = \sum_{n \ge 0} t^n$.
Thus, for any $n \ge 2j = 0$, we have $f_{n,0}(-1) = 1$.

If $j \ge 1$, 
then $F_j(-1,t) = t^j (1-t)^{j-1}$ is a polynomial in $t$ of degree \(2j-1\).
Consequently, the coefficient of $t^n$ must be $0$ for $n\ge 2j$.
 This completes the proof.
\end{proof}

Proposition \ref{prop:g1} implies $x=0$ is a root of $g_{n,j}(x)$ for $j\ge 1,n\ge 2j$.
Moreover, Proposition~\ref{prop:g1} and the definition of \(g_{n,j}(x)\) imply the following identities.

\begin{corollary}
	For \(n\ge 0\), we have
	\[
	\sum_{k=0}^{\lfloor n/2\rfloor}
	(-1)^k \binom{n-k}{k}C_{n-k}=1.
	\]
	For \(j\ge 1\) and \(n\ge 2j\), we have
	\[
	\sum_{k=0}^{\lfloor n/2\rfloor}
	(-1)^k \binom{n-k}{k}C_{n-k-j}=0.
	\]
\end{corollary}

\section{Conclusion}
A sequence $(f_0(x), f_1(x), \dots, f_m(x))$ of real-rooted polynomials is called \textit{interlacing} 
if $f_i(x)$ interlaces $f_j(x)$ for all $0 \le i < j \le m$. 
Let $(f_0(x), f_1(x), \dots, \\f_m(x))$ be an interlacing sequence of real-rooted polynomials with positive leading coefficients.
It is well known \cite[Lemma 2.6]{BB10}  that 
every nonnegative linear combination $f(x)$ of $f_0(x), f_1(x), \cdots, f_m(x)$ is real-rooted and satisfies
\[
f_0(x)\preceq f(x)\preceq f_m(x).
\]

Ehrenborg, Hetyei and Readdy showed that the toric \(g\)-polynomial of an \(n\)-dimensional simple polytope \(P\) can be written as
\begin{equation*}
g(P, x) = \sum_{j=0}^{\lfr{n/2}} \gamma_j \cdot g_{n,j}(x),
\end{equation*}
where $(\gamma_0, \gamma_1, \dots, \gamma_{\lfr{n/2}})$ is the $\gamma$-vector of the simple polytope $P$.
They also proposed the following real-rootedness conjecture for the the \(g\)-polynomials of cyclohedron $\mathcal{C}_n$ and chordal nestohedra $P_{\mathcal{B}}$.
\begin{conjecture}(\cite[Conjecture 11.2]{EHR25})\label{EHR2}
The $g$-polynomials of the $n$-dimensional cyclohedron and chordal nestohedra are real-rooted.  
\end{conjecture}

The $\gamma$-vector of the $n$-dimensional cyclohedron $\mathcal{C}_n$ is given by \cite[Proposition 11.15]{PRW08} 
\[
\gamma_j(\mathcal{C}_n)=\binom{n}{2j}\binom{2j}{j}.
\]
For an $n$-dimensional chordal nestohedron $P_{\mathcal{B}}$, 
the entry $\gamma_j$ is the number of right-adjusted Foata-Strehl trees representing a $\mathcal{B}$-permutation and having $j$ forks \cite[Corollary 10.5]{EHR25}.
In both cases, the $\gamma$-vector is nonnegative.

At the suggestion of Christos Athanasiadis, 
we used Mathematica to verify that the sequence
\[
\bigl(g_{n,0}(x),g_{n,1}(x),\ldots,g_{n,\lfloor n/2\rfloor}(x)\bigr)
\]
is interlacing for all \(n\le 120\).
Although this verification is finite, 
the data suggests that the same property may persist for all \(n\).
Athanasiadis further suggested that we formulate 
the following conjecture.

\begin{conjecture} \label{Conjg}
The polynomial sequence $(g_{n,0}(x),g_{n,1}(x),\dots,g_{n,\lfr{n/2}}(x))$ is interlacing for every $n\ge 0$.
\end{conjecture}

If Conjecture~\ref{Conjg} holds, 
then \cite[Lemma 2.6]{BB10} implies that the toric \(g\)-polynomial of every \(n\)-dimensional simple polytope with nonnegative \(\gamma\)-vector is real-rooted. 
In particular, it would imply Conjecture~\ref{EHR2}.

\medskip
\noindent
\textbf{Acknowledgements}. 
The author wishes to thank Christos Athanasiadis for suggesting the problem, helpful discussions, and valuable comments and suggestions on an earlier version of this manuscript, and Yanxin Liu and Xue Yan for helpful discussions.
The author acknowledges support from the China Scholarship Council (No. 202306060164) for the visit to the Department of Mathematics of the National and Kapodistrian University of Athens during the academic year 2025--26.

\bibliography{ref}

@book{Ath26,
	author    = {Christos A. Athanasiadis},
	title     = {Discrete Mathematics: A Combinatorial Approach},
	series    = {Undergraduate Texts in Mathematics},
	publisher = {Springer},
	year      = {2026},
}

@article {BB10,
AUTHOR = {Borcea, Julius Bogdan and Br\"and\'en, Petter},
TITLE = {Multivariate {P}\'olya-{S}chur classification problems in the
{W}eyl algebra},
JOURNAL = {Proc. Lond. Math. Soc. (3)},
FJOURNAL = {Proceedings of the London Mathematical Society. Third Series},
VOLUME = {101},
YEAR = {2010},
NUMBER = {1},
PAGES = {73--104},
}

@incollection {Bra15,
AUTHOR = {Br\"and\'en, Petter},
TITLE = {Unimodality, log-concavity, real-rootedness and beyond},
BOOKTITLE = {Handbook of enumerative combinatorics},
SERIES = {Discrete Math. Appl. (Boca Raton)},
PAGES = {437--483},
PUBLISHER = {CRC Press, Boca Raton, FL},
YEAR = {2015},
}

@article{EHR25,
  title={The Toric $g$-Vector of Nestohedra},
  author={Ehrenborg, Richard and Hetyei, G{\'a}bor and Readdy, Margaret},
  journal={arXiv preprint arXiv:2511.04815},
  year={2025}
}

@article{Gal05,
AUTHOR = {Gal, \'Swiatos\l aw R.},
TITLE = {Real root conjecture fails for five- and higher-dimensional
spheres},
JOURNAL = {Discrete Comput. Geom.},
FJOURNAL = {Discrete \& Computational Geometry. An International Journal
of Mathematics and Computer Science},
VOLUME = {34},
YEAR = {2005},
NUMBER = {2},
PAGES = {269--284},
}

@article{Het12,
  title={A second look at the toric h-polynomial of a cubical complex},
  author={Hetyei, G{\'a}bor},
  journal={Annals of Combinatorics},
  volume={16},
  number={3},
  pages={517--541},
  year={2012},
  publisher={Springer}
}

@article{Kar04,
  title={Hard {L}efschetz theorem for nonrational polytopes},
  author={Karu, Kalle},
  journal={Inventiones mathematicae},
  volume={157},
  number={2},
  pages={419--447},
  year={2004},
  publisher={Springer}
}

@article {LW07,
    AUTHOR = {Liu, Lily L. and Wang, Yi},
     TITLE = {A unified approach to polynomial sequences with only real
              zeros},
   JOURNAL = {Adv. in Appl. Math.},
  FJOURNAL = {Advances in Applied Mathematics},
    VOLUME = {38},
      YEAR = {2007},
    NUMBER = {4},
     PAGES = {542--560},
}

@article{PRW08,
  title={Faces of generalized permutohedra},
  author={Postnikov, Alex and Reiner, Victor and Williams, Lauren},
  journal={Documenta Mathematica},
  volume={13},
  pages={207--273},
  year={2008}
}

@incollection {Sta87,
    AUTHOR = {Stanley, Richard},
     TITLE = {Generalized {$H$}-vectors, intersection cohomology of toric
              varieties, and related results},
 BOOKTITLE = {Commutative algebra and combinatorics ({K}yoto, 1985)},
    SERIES = {Adv. Stud. Pure Math.},
    VOLUME = {11},
     PAGES = {187--213},
 PUBLISHER = {North-Holland, Amsterdam},
      YEAR = {1987},
}

@book {Sta12,
    AUTHOR = {Stanley, Richard P.},
     TITLE = {Enumerative combinatorics. {V}olume 1},
    SERIES = {Cambridge Studies in Advanced Mathematics},
    VOLUME = {49},
   EDITION = {Second},
 PUBLISHER = {Cambridge University Press, Cambridge},
      YEAR = {2012},
     PAGES = {xiv+626},
}

\end{document}